\documentclass[11pt]{amsart}

\newtheorem{theorem}{Theorem}[section]

\newtheorem{proposition}[theorem]{Proposition}

\theoremstyle{definition}

\usepackage{tikz-cd}
\usepackage{mathrsfs}

\usepackage[all,cmtip]{xy}

\newcommand{\C}{\mathbb{C}}

\newcommand{\G}{\mathbb{G}}
\newcommand{\D}{\mathbb{D}}

\newcommand{\leafBid}{\mathscr{F}}

\newcommand{\autd}{Aut (\mathbb{D})}

\theoremstyle{remark}

\numberwithin{equation}{section}

\begin{document}

\title[A characterization of the bidisc]{A characterization of the bidisc by a subgroup of its automorphism group}

\author{Anindya Biswas}
\address{Department of Mathematics, Indian Institute of Science, Bangalore 560012, India.}
\email{anindyab@iisc.ac.in}
\thanks{}

\author{Anwoy Maitra}
\address{Harish-Chandra Research Institute, Chhatnag Road, Jhunsi, Prayagraj 211019, India.}
\email{anwoymaitra@hri.res.in (Corresponding author)}
\thanks{}

\subjclass[2010]{Primary 32M05, Secondary 32V15, 32V40}

\keywords{}

\date{}

\dedicatory{}

\begin{abstract}
	
		We make a connection between the structure of the bidisc and
	a distinguished subgroup of its automorphism group. The
	automorphism group of the bidisc, as we know,
	is of dimension six and acts transitively. We observe that it contains a subgroup that is isomorphic to the automorphism group of the open unit disc and this subgroup partitions
	the bidisc into a complex curve and a family of strongly pseudo-convex hypersurfaces that are non-spherical as 
	CR-manifolds. Our work reverses this process and shows that  any $2$-dimensional Kobayashi-hyperbolic manifold whose
	automorphism group (which is known, from the general theory,
	to be a Lie group) has
	a $3$-dimensional subgroup that is non-solvable (as a Lie
	group) and that acts on the manifold to produce a collection
	of orbits possessing essentially the characteristics of the
	concretely known collection of orbits mentioned above, is
	biholomorphic to the bidisc. 
	
	The distinguished subgroup is interesting in its own right. It turns out that if we consider any subdomain of the bidisc that is a union of a proper sub-collection of the collection of orbits mentioned above, then the automorphism group of this subdomain can be expressed very simply in terms of this
	distinguished subgroup.
	
\end{abstract}
	\keywords{Bidisc, automorphism group, classification of
	hyperbolic manifolds, subdomains of the bidisc}

\maketitle

\section{Introduction}
\label{intro}
This paper is concerned with presenting a characterization of the bidisc. Namely, we will
pick out the bidisc from among all 2-dimensional Kobayashi-hyperbolic complex manifolds
by making some assumptions about the automorphism group of such a manifold and about its
action on the manifold.

There has been work---see \cite{WongBid}---characterizing the bidisc by picking it out
from among all bounded
domains in $\mathbb{C}^2$ by making some extrinsic assumptions about the action of the
automorphism group of the domain on the domain (the assumptions are extrinsic in that
they deal with the convergence of certain sequences to boundary points). In our work, we
will only make intrinsic assumptions about the automorphism group and its action on the
manifold. Our work is heavily dependent on the seminal work \cite{Isaev23}, in which all
2-dimensional Kobayashi-hyperbolic complex manifolds with (real) 3-dimensional
automorphism group have been classified. The current work is also an outgrowth of the
earlier work \cite{BhatBisMai} of the two present authors (joint with T.~Bhattacharyya).

The bidisc (along with the Euclidean unit ball in $\C^2$) is a pre-eminent model domain
that has been studied intensively since the earliest days of several complex variables.
It is, therefore, of interest to have characterizations of it. Roughly speaking, our
main result says that any 2-dimensional Kobayashi-hyperbolic complex manifold whose
automorphism group (which is known to be a Lie group) is assumed to have a certain
non-solvable Lie \emph{subgroup} under the action of which all the orbits, except
only one, are assumed to be strongly pseudoconvex hypersurfaces, and
the sole remaining orbit is assumed to be a complex curve, must be
biholomorphic
to the bidisc if one makes the further assumption that these hypersurfaces are of a
specific type. To be more precise, we make the assumption that these hypersurfaces
are CR-equivalent to certain standard hypersurfaces that occur in E. Cartan's famous
classification of connected homogeneous strongly pseudoconvex 3-dimensional CR 
manifolds \cite{Cartan} (see also \cite{IsaevAnalogues}). Equivalently, we make the
assumption that the hypersurfaces are CR-equivalent to certain hypersurfaces (see 
(\ref{eq:leafBid}) and Proposition~\ref{CREquivalenceOfOrbits}) one
obtains
when a particular distinguished subgroup (see \eqref{G_D})
of $Aut(\mathbb{D}^2)$ acts on $\D^2$. Along the way, we also obtain characterizations of some
distinguished sub-domains of $\mathbb{D}^2$ (see Theorem~\ref{CharOfD^2_r}). 

A noteworthy feature of our work is that we only make assumptions
about some subgroup of the automorphism group of the
initially arbitrary manifold that we start out with. To the best
of our knowledge, this is a new approach. 

The plan of this work is as follows: in Section~\ref{PreliResult}, we
introduce the distinguished subgroup of $Aut(\mathbb{D}^2)$ and the distinguished
sub-domains of $\mathbb{D}^2$ mentioned above, and connect them to certain model
domains exhibited in \cite{Isaev23} and their automorphism groups. After that we investigate the action of some subgroups of 
$Aut(\mathbb{B}_2)$ on $\mathbb{B}_2$, where $\mathbb{B}_2$ denotes the Euclidean
unit ball in $\mathbb{C}^2$
(this is needed for the proof of our main theorem). Finally, in 
Section~\ref{MainResult}, we prove our main result characterizing
$\mathbb{D}^2$.      
\section{Preliminary results}\label{PreliResult}

\subsection{Some domains in $\D^2$ and their automorphism groups}\label{sec:dom_D^2_aut_grp}

We start with $\D^2$ and its automorphism group which is given by
\begin{align*}
Aut(\D^2)=\{(\varphi_1\times \varphi_2), (\varphi_1\times \varphi_2) \circ \sigma : \varphi_1, \varphi_2\in \autd \},
\end{align*}
where $(\varphi_1\times \varphi_2)(z_1,z_2)=(\varphi_1(z_1), \varphi_2(z_2))$ and $\sigma(z_1,z_2)=(z_2, z_1)$. This can be found in \cite{Rudin poly}. In \cite{BhatBisMai}, a closed subgroup of $Aut(\D^2)$ has been identified and its action on $\D^2$ has been observed. The subgroup is 
\begin{align}\label{G_D}
G_\D =\{\Phi_{\varphi}: \varphi\in \autd \},
\end{align}
where $\Phi_{\varphi}(z_1,z_2)=(\varphi(z_1),\varphi(z_2))$. This group does not act transitively and its orbits are 
\begin{align}\label{eq:leafBid}
\mathscr{F}_a =\Big\{(z_1,z_2)\in \D^2 :\Big|\frac{z_1 -z_2}{1-\overline{z_1}z_2}\Big|=a\Big\},\,\,a\in[0,1).
\end{align}
From \cite{BhatBisMai} we know that $\mathscr{F}_0$ is a complex curve and $\mathscr{F}_a$ is a strongly pseudoconvex hypersurface for each $a\in (0,1)$.

In \cite{Isaev_n_n2-1} and \cite{Isaev23}, the author listed some collections of domains which are given by
\begin{align}\label{D2st}
\mathcal{D}^{(2)}_{s,t}=&\{(z_1,z_2,z_3)\in \mathbb{C}^3:s<|z_1|^2 +|z_2|^2 -|z_3|^2<t,z_1 ^2 +z_2 ^2 -z_3^2=1,\\ \notag
& Im(z_2(\overline{z_1}+\overline{z_3}))>0\}, 1\leq s<t\leq\infty
\end{align}
and
\begin{align}\label{D2s}
\mathcal{D}^{(2)}_s =&\{(1:z_1:z_2:z_3)\in \mathbb{CP}^3:s<|z_1|^2 +|z_2|^2 -|z_3|^2,z_1 ^2 +z_2 ^2 -z_3^2=1,\\ \notag
&Im(z_2(\overline{z_1}+\overline{z_3}))>0\}
\cup\{(0:z_1:z_2:z_3)\in \mathbb{CP}^3:z_1 ^2 +z_2 ^2 -z_3^2=
0,\\ \notag
&Im(z_2(\overline{z_1}+\overline{z_3}))>0\}, 1<s.
\end{align}
Here $\mathcal{O}^{(2)} =\{(0:z_1:z_2:z_3)\in \mathbb{CP}^3:z_1 ^2 +z_2 ^2 -z_3^2=0,
Im(z_2(\overline{z_1}+\overline{z_3}))>0\}$ is a complex curve and we have $\mathcal{D}^{(2)}_s=\mathcal{D}^{(2)}_{s,\infty} \cup \mathcal{O}^{(2)}$.

The connected components containing the identity of the
automorphism groups of the domains (\ref{D2st}) and (\ref{D2s})
are 
all equal to $SO_{2,1}(\mathbb{R})^0$ (the connected component of
$SO_{2,1}(\mathbb{R})$ containing the 
identity---hereinafter, for $K$ an arbitrary topological
group, we shall use $K^0$ to denote the connected component
of $K$ containing the identity) with the following action 

\begin{align}\label{AutActionD^2_1,infty}
\begin{pmatrix}
z_1\\
z_2\\
z_3
\end{pmatrix}\mapsto A\begin{pmatrix}
z_1\\
z_2\\
z_3
\end{pmatrix}
\end{align}
and
\begin{align}\label{AutActionD^2_s}
\begin{pmatrix}
s\\
z_1\\
z_2\\
z_3
\end{pmatrix}\mapsto \begin{pmatrix}
1 &\mathbf{0}\\
\mathbf{0} &A
\end{pmatrix}\begin{pmatrix}
s\\
z_1\\
z_2\\
z_3
\end{pmatrix}, \; s=0,1,
\end{align}
respectively.
We write $\mathcal{R}^{(1)}$ to denote this group
($SO_{2,1}(\mathbb{R})^0$ with the action described above). The
$\mathcal{R}^{(1)}$-orbits of (\ref{D2st}) are given by 
\begin{align}\label{eta2alpha}
\eta^{(2)}_\alpha=\Bigg \{(z_1,z_2,z_3)\in \mathcal{D}^{(2)}_{s,t}:|z_1|^2 +|z_2|^2 -|z_3|^2=\sqrt{\frac{\alpha +1}{2}}\Bigg\},s<\alpha <t.
\end{align}
It is known from \cite{IsaevAnalogues} that the hypersurfaces $\eta^{(2)}_\alpha$ are strongly pseudoconvex,
non-spherical and pairwise CR-nonequivalent.
The $\mathcal{R}^{(1)}$-orbits of (\ref{D2s}) are $\eta^{(2)}_\alpha$, $\alpha>s$, together with the complex curve $\mathcal{O}^{(2)}$.

Now we shall connect these domains $\mathcal{D}^{(2)}_{s,t}$ and $\mathcal{D}^{(2)}_{s}$ with domains in $\D^2$. Consider the map $J:\D^2 \rightarrow\mathbb{CP}^3$ defined by
\begin{align}\label{Definition of J}
J(z,w)=(z-w:1-zw:i(1+zw):-i(z+w)).
\end{align}
In the article \cite{BhatBisMai}, the authors showed that $J$ is a biholomorphism from $\D^2$  onto $\mathcal{D}^{(2)} _1$. It maps $D=\{(z,z):z\in \D\}$ onto $\mathcal{O}^{(2)}$. From this map we derive
$H:\D^2 -D \rightarrow\mathbb{C}^3$ defined by
\begin{align}\label{DefnOfH}
H(z,w)=\Big(\frac{1-zw}{z-w},i \frac{1+zw}{z-w},-i \frac{z+w}{z-w}\Big).
\end{align}

\begin{proposition}\label{CREquivalenceOfOrbits}
	$H$ maps $\D^2 -D$ onto $\mathcal{D}^{(2)}_{1,\infty}$ biholomorphically. Moreover, for any $a\in(0,1)$, $H$ gives rise to a CR-isomorphism between $\mathscr{F}_a$ and $\eta^{(2)}_\alpha$ where $\alpha=\frac{8}{a^4} -\frac{8}{a^2}+1$.
\end{proposition}
\begin{proof}
	The proof that $H$ is a biholomorphism between $\D^2 -D$ and $\mathcal{D}^{(2)}_{1,\infty}$ follows from the observations that the $J$ above is a biholomorphism that maps $D=\{(z,z):z\in \D\}$ onto $\mathcal{O}^{(2)}$ and that $\mathcal{D}^{(2)}_1=\mathcal{D}^{(2)}_{1,\infty} \cup \mathcal{O}^{(2)}$. The remaining part of the proposition follows from the observation that $H$ maps $\mathscr{F}_a$ onto $\eta^{(2)}_\alpha$, where $\alpha$ and
	$a$ are related as described above.
\end{proof}
The maps $J$ and $H$ open up a way to finding 
distinguished subdomains of $\D^2$ and $\D^2 -D$.
For $r,s,t$ with $1\leq s<t\leq \infty$ and $1<r$, consider the domains $J^{-1}(\mathcal{D}^{(2)}_r)$ and $H^{-1}(\mathcal{D}^{(2)}_{s,t})$ in $\D^2$. From the definitions of $\mathcal{D}^{(2)}_r$ and $\mathcal{D}^{(2)}_{s,t}$, it is easy to see that

\begin{align*}
H^{-1}(\mathcal{D}^{(2)}_{s,t})=\Bigg\{(z_1,z_2)\in \D^2: \sqrt{\frac{2}{t+1}}<\Big|\frac{z_1 -z_2}{1-\overline{z_1}z_2}\Big|< \sqrt{\frac{2}{s+1}}\Bigg\}
\end{align*}
and 
\begin{align}\label{CopyOfD2_r}
J^{-1}(\mathcal{D}^{(2)}_r)=\Bigg\{(z_1,z_2)\in \D^2: \Big|\frac{z_1 -z_2}{1-\overline{z_1}z_2}\Big|< \sqrt{\frac{2}{r+1}}\Bigg\}.
\end{align}
So it is natural to introduce domains $\D^2 _r$ and $\D^2 _{s,t}$ defined by
\begin{align*}
\D^2 _r =\Big\{(z_1,z_2)\in\D^2 :\Big|\frac{z_1 -z_2}{1-\overline{z_1}z_2}\Big|<r \Big\}, r\in (0,1)
\end{align*}
and 
\begin{align*}
\D^2 _{s,t} =\Big\{(z_1,z_2)\in\D^2 :s<\Big|\frac{z_1 -z_2}{1-\overline{z_1}z_2}\Big|<t \Big\}, 0\leq s<t\leq 1.
\end{align*}
Note that $\D^2_1 =\D^2$, if we define $\D^2_1$ in a manner
similar to the above. 
Now we shall show that $Aut(\D ^2 _{r})$ and $Aut(\D ^2 _{s,t})$ can be expressed in terms of $G_\D$. Before going into that result, we want to focus on a few facts. 
\begin{enumerate}
	\item $H^{-1}\mathcal{R}^{(1)}H=G_\D$ and $H^{-1}(-I_3)H=\sigma$, where $\sigma (z_1,z_2)=(z_2,z_1)$.\\
	To see this, recall (look at the sentence immediately 
	following \eqref{AutActionD^2_s}) that $\mathcal{R}^{(1)}$ is isomorphic to
	$SO_{2,1}(\mathbb{R})^0$. 
	We know from \cite{Isaev23} that $\mathcal{R}^{(1)}=G(\mathcal{D}^{(2)}_{1,\infty})$ (the
	latter denotes the connected identity-component of $Aut(\mathcal{D}^{(2)}_{1,\infty})$---hereinafter, for $M$
	an arbitrary hyperbolic complex manifold, we shall use
	$G(M)$ to denote $Aut(M)^0$). Therefore $\mathcal{R}^{(1)}$ 
	is a 3-dimensional Lie
	group. Since $H$ is a biholomorphism between $\D^2 _{0,1}$ and $\mathcal{D}^{(2)}_{1,\infty}$, $G(\D^2 _{0,1})=H^{-1}\mathcal{R}^{(1)}H$ is also a 
	3-dimensional Lie group. Recall the subgroup $G_\D$ (see (\ref{G_D})) of 
	$\text{Aut}(\D^2)$. It is clearly a subgroup of $G(\D^2 _{0,1})$ and a
	3-dimensional Lie group in its own right. But we already know that
	$G(\D^2 _{0,1})$ is 3-dimensional. Therefore, 
	$G_{\D}$ is a connected 3-dimensional subgroup of the connected
	3-dimensional Lie group $H^{-1}\mathcal{R}^{(1)}H$ and,
	consequently,
	coincides with the latter, i.e., $H^{-1}\mathcal{R}^{(1)}H = 
	G_{\D}$ (an open connected subgroup $K$ of a Lie group $G$
	is equal to the connected component of $G$ containing the identity,
	i.e., $K=G^0$---see, for example, Lemma~7.12 on page $156$ of \cite{Lee Smooth}). The fact that $H^{-1}(-I_3)H=\sigma$ follows easily from the definition of $H$.\\
	
	\item $Aut_{CR} (\mathscr{F}_a)=\{\Phi, \Phi\circ \sigma :\Phi \in G_\D\}$ for all $a\in (0,1)$.\\
	This is a consequence of the fact that for any $\alpha\in (1,\infty)$, $Aut_{CR} (\eta^{(2)} _\alpha)$ is generated by $\mathcal{R}^{(1)}$ and $-I_3$ (see page $669$ of \cite{Isaev23}), and $\mathscr{F}_a =H^{-1}(\eta^{(2)} _\alpha)$ for suitable $a$ and $\alpha$.\\
	
	\item For every holomorphic automorphism $\Phi$ of
	$\D^2_r$ (resp., $\D^2_{s,t}$) and for every $a \in (0,r)$
	(resp., $(s,t)$), $\Phi$ maps $\mathscr{F}_a$ bijectively to
	itself and hence $\Phi|_{\mathscr{F}_a}$ is a CR-automorphism
	of $\mathscr{F}_a$.

\end{enumerate}

\begin{theorem}\label{Aut_Of_D2_st}
	For $r,s,t$ with $r<1$ and $0\leq s<t\leq 1$, $Aut(\D^2 _r)$ and $Aut(\D^2 _{s,t})$ are equal to $\{\Phi, \Phi\circ \sigma :\Phi \in G_\D\}$, where $\sigma (z_1,z_2)=(z_2,z_1)$.
\end{theorem}

\begin{proof}
	We start with an arbitrary automorphism $\Phi$ of $\D^2_{r}$
	(resp., $\D^2_{s,t}$). We fix $a \in (0,r)$ (resp., $(s,t)$)
	arbitrarily. By the three observations above,
	we know that there exists $\varphi \in
	Aut(\D)$ such that $\Phi|_{\leafBid_a}
	= \Phi_{\varphi}$ or $\Phi_{\varphi} \circ \sigma$.
	Now note that $\Phi_{\varphi}$ (or $\Phi_{\varphi} \circ
	\sigma$, as the case may be) is itself an automorphism of
	$\D^2_{r}$
	(resp., $\D^2_{s,t}$). At any rate, $\Phi$
	and $\Phi_{\varphi}$ (or $\Phi_{\varphi} \circ 
	\sigma$, as the case may be) are both holomorphic maps on
	$\D^2_{r}$
	(resp., $\D^2_{s,t}$) that agree on the subset $\leafBid_a$,
	which is a real 3-dimensional submanifold of $\D^2_{r}$
	(resp., $\D^2_{s,t}$). We know, from Corollary 2 of 
	Section~2.6 of Chapter~1 of \cite{Chirka}, that if the set of points of
	$\D^2_{r}$ (resp., $\D^2_{s,t}$) at which two
	holomorphic functions defined thereon agree (which is necessarily an
	analytic subset of $\D^2_{r}$ (resp., $\D^2_{s,t}$))
	has Hausdorff dimension greater than $2\times2-2=2$, then the two
	functions are identically equal. In our case, both the first and
	second component functions of $\Phi$ and $\Phi_{\varphi}$ (or
	$\Phi_{\varphi} \circ \sigma$, as the case may be) agree
	on the subset $\leafBid_a$ of $\D^2_{r}$ (resp.,
	$\D^2_{s,t}$), which has Hausdorff dimension 3, being a real
	3-dimensional submanifold of $\D^2_{r}$ (resp., 
	$\D^2_{s,t}$). Therefore, both the pairs are identically
	equal to one another, and this implies that
	$\Phi=\Phi_{\varphi}$ (or $\Phi_{\varphi} \circ
	\sigma$, as the case may be). This completes the proof.
	
\end{proof}

\textbf{Remark:} Another proof of Theorem \ref{Aut_Of_D2_st} can be given using the properties of the
symmetrized bidisc. We shall give a sketch of the proof.

\begin{enumerate}
	\item It suffices to show that if $\Phi\in Aut(\D^{2}_{s,t})$ (or $Aut(\D^{2}_{r})$) and $\Phi|_{\mathscr{F}_a} =id$ for some $a>0$, then $\Phi=id$. (The reason why this suffices is
	that, as we know, if we fix $a \in (s,t)$ (resp., $(0,r)$)
	arbitrarily, then there exists $\varphi \in
	Aut(\D)$ such that $\Phi|_{\leafBid_a}
	= \Phi_{\varphi}$ or $\Phi_{\varphi} \circ \sigma$. Therefore
	we can consider the automorphism $\Phi_{\varphi} \circ
	\Phi^{-1}$ (or $(\Phi_{\varphi}\circ\sigma)\circ\Phi^{-1}$,
	as the case may be) and attempt to conclude that it is equal
	to the identity.)
	\item Define $sym :\D^2 \rightarrow\G$ by $sym (z_1,z_2)=(z_1 +z_2,z_1z_2)$, $\G_r =sym (\D^2 _r)$ and $\G_{s,t}=sym (\D^2 _{s,t})$. Using the map in Theorem $3.2$ in \cite{BhatBisMai} and properties stated in \cite{Isaev23} (page 654--655), we see that $Aut(\G_{r})=Aut (\G_{s,t})=\{H_\varphi :\varphi \in \autd\}$ where $H_\varphi\circ sym (z_1,z_2)=sym (\varphi(z_1),\varphi(z_2))$.
	\item Define $\Psi_\Phi :\G_{s,t} \rightarrow\C^2$ by $\Psi_\Phi \circ sym (z_1,z_2)=sym \circ \Phi (z_1,z_2)$. It is easy to see that $\Psi_\Phi$ is well-defined and is a 
	holomorphic automorphism of $\G_{s,t}$ (respectively, of $\G_{r}$, when defined similarly). So $\Psi_\Phi=H_\varphi$ for some $\varphi\in \autd$. Using $\Phi|_{\mathscr{F}_a} =id$, we get $\varphi =id\in \autd$.
	\item So for any $(z_1,z_2)$, $\Phi (z_1,z_2)$ is either $(z_1,z_2)$ or $(z_2,z_1)$. A simple argument involving sequences shows that $\Phi=id$ in a neighbourhood of $(a,0)$ and this will imply the conclusion of Theorem \ref{Aut_Of_D2_st}.
\end{enumerate}

In our article, we shall use the domains $\D^2 _r$ to prove our main result. By \cite{Isaev23}, it is a hyperbolic 2-manifold with 3-dimensional automorphism group. Our next result is a characterization of $\D^2 _r$.
\begin{theorem}\label{CharOfD^2_r}
	Suppose $M$ is a hyperbolic 2-manifold with 3-dimensional automorphism group. Let $G(M) := Aut(M)^0$. Suppose that $M$ has an orbit under the action of $G(M)$ that is a complex curve, that 
	there exist an $a \in (0,1)$ and a 3-dimensional orbit $O$ of $M$ under $G(M)$ that is CR-equivalent to $\mathscr{F}_a$ and that all 
	remaining 
	orbits are strongly pseudoconvex real hypersurfaces. Then there exists an
	$r \in(0,1)$ such that $M$ is biholomorphic to $\D^{2}_{r}$.
\end{theorem}
\begin{proof}
	Our theorem follows, with very little effort, from Isaev's work.
	It follows from the proof of \cite[Theorem~5.1]{Isaev23} that if $M$ is a hyperbolic 2-manifold with 3-dimensional automorphism group having an orbit under the action of $G(M)$
	that is a complex curve, such that all its remaining orbits
	are strongly pseudoconvex (3-dimensional) real hypersurfaces and
	such that one of these 3-dimensional orbits is CR-equivalent to
	$\eta^{(2)}_{\alpha}$ for some $\alpha \in (1,\infty)$, then $M$ is biholomorphic to $\mathcal{D}^{(2)}_s$ for some $s \in (1,\infty)$. But we have already seen that $\eta^{(2)}_{\alpha}$ is
	CR-equivalent to $\mathscr{F}_a$, where
	\[
	\alpha=\frac{8}{a^4} -\frac{8}{a^2}+1,
	\]
	and $\mathcal{D}^{(2)}_s$ is biholomorphic with some $\D^{2}_{r}$. Therefore, we can use the observation that we stated above (see (\ref{CopyOfD2_r})) to conclude that $M$ is biholomorphic to $\D^{2}_{r}$ for some $r<1$. 
\end{proof}

Now suppose that, with $M$ as above, we have the following:

\begin{enumerate}
	\item Exactly one $G(M)$-orbit is a complex curve.
	\item The remaining orbits are strongly pseudoconvex real hypersurfaces.
	\item There exist a sequence $\{a_n\} \subset (0,1)$ converging to $1$, and, for each $n$, an orbit $\mathcal{O}_n$ in $M$ which is CR-equivalent to $\mathscr{F}_{a_n}$.
\end{enumerate}
Under these conditions, $M$ cannot be biholomorphic with $\D^2 _r$ for any $r\in(0,1)$. The reason behind this claim is the following: If $b_1, b_2\in (0,1)$ are two distinct numbers, then $\mathscr{F}_{b_1}$ and $\mathscr{F}_{b_2}$ are not CR-equivalent
(they are CR-equivalent to $\eta^{(2)}_{\beta_1}$ and 
$\eta^{(2)}_{\beta_2}$, respectively, where
\[
\beta_j = \frac{8}{b_j^4}-\frac{8}{b_j^2}+1, \; j=1,2;
\]
hence $\beta_1 \neq \beta_2$ and so, by the remark made 
immediately after \eqref{eta2alpha}, $\mathscr{F}_{b_1}$ and
$\mathscr{F}_{b_2}$ are CR-nonequivalent.) If $M$ is
biholomorphic with $\D^2 _r$ for some $r\in(0,1)$, then,
considering the action of $G_\D$ on $\D^2_r$ and the orbits of
$\D^2_r$ under $G_{\D}$, we are led to conclude that there exists
an orbit of $\D^2_r$ under $G_\D$ that is CR-equivalent to
some $\mathscr{F}_{a_n}$, where $a_n\in (r,1)$. But, taking into
consideration the CR-nonequivalence of the
$\mathscr{F}_{a}$'s, this is impossible.  

\subsection{Action of $SU_{1,1} (\mathbb{C})$ and $O_{2,1} (\mathbb{R})$ on the open unit ball}\label{ActionOnTheBall}

Let $\mathbb{B}_2$ denote the open Euclidean unit ball in $\mathbb{C}^2$. Its automorphism group has a connection with the Lie group $SU_{2,1} (\mathbb{C})$, which acts on $\mathbb{B}_2$ in the following way
$$\begin{pmatrix}
a_{1} &a_{2} &a_{3}\\
b_{1} &b_{2} &b_{3}\\
c_{1} &c_{2} &c_{3}
\end{pmatrix}\cdot(u,v)=\frac{(a_1 u + a_2 v + a_3,b_1 u +b_2 v +b_3)}{c_1 u +c_2 v + c_3}.$$
In this subsection we consider $O_{2,1} (\mathbb{R})$ and a copy of $SU_{1,1} (\mathbb{C})$ as subgroups of $SU_{2,1} (\mathbb{C})$ and study their actions on $\mathbb{B}_2$.

First we investigate the action of $SU_{1,1} (\mathbb{C})$. Consider
\[
\begin{pmatrix}
\alpha & \beta \\
\overline{\beta} & \overline{\alpha}
\end{pmatrix} \mapsto
\begin{pmatrix}
1 &0 &0\\
0 &\alpha &\beta\\
0 &\overline{\beta} &\overline{\alpha}
\end{pmatrix}
\]
where $|\alpha|^2 -|\beta|^2=1$. This is an embedding of $SU_{1,1} (\mathbb{C})$ in $SU_{2,1} (\mathbb{C})$. 
We shall identify $SU_{1,1} (\mathbb{C})$ with the subgroup
\begin{align*}
\Bigg\{\begin{pmatrix}
1 &0 &0\\
0 &\alpha &\beta\\
0 &\overline{\beta} &\overline{\alpha}
\end{pmatrix}: \alpha,\beta\in \mathbb{C},|\alpha|^2 -|\beta|^2=1 \Bigg\}
\end{align*}
of $SU_{2,1} (\mathbb{C})$ and observe its action on $\mathbb{B}_2$.

\begin{proposition}\label{ActionOfSU11}
	The $SU_{1,1} (\mathbb{C})$-orbits of $\mathbb{B}_2$ are given by $SU_{1,1} (\mathbb{C})\cdot (t,0),\,\,t\in [0,1)$. $SU_{1,1} (\mathbb{C})\cdot (0,0)$ is a complex curve and, for $t\in (0,1)$, $SU_{1,1} (\mathbb{C})\cdot (t,0)$ is CR-equivalent with the sphere, that is, the topological boundary of $\mathbb{B}_2$.
\end{proposition}

\begin{proof}
	Let $(u,v)\in \mathbb{B}_2$. For 
	$$A=\begin{pmatrix}
	1 &0 &0\\
	0 &\alpha &\beta\\
	0 &\overline{\beta} &\overline{\alpha}
	\end{pmatrix},\,\,\text{we have}\,\,A\cdot(u,v)=\frac{(u,\alpha v +\beta)}{\overline{\beta}v +\overline{\alpha}}.$$
	Note that $v\in \D$ and $\frac{\alpha v +\beta}{\overline{\beta}v +\overline{\alpha}}$ is the action of $\begin{pmatrix}
	\alpha &\beta\\
	\overline{\beta} &\overline{\alpha}
	\end{pmatrix}$ on $v$. Since the action of $SU_{1,1} (\mathbb{C})$ (considered as $2\times 2$ matrices) is transitive on $\D$, a little calculation gives that with proper choice of $A$, we can find a $t\in [0,1)$ such that $A\cdot (u,v)=(t,0)$. If for a fixed $(u,v)$, there are two elements of $(0,1)$, say $t_1,t_2$, such that $(t_1,0)$ and
	$(t_2,0)$ both belong to the $SU_{1,1}(\C)$-orbit of
	$(u,v)$,	
	then that means that there exists an element of 
	$SU_{1,1}(\C)$ carrying $(t_1,0)$ to $(t_2,0)$, hence, from
	the above, that  
	there exist complex numbers $\alpha, \beta$ such that $\frac{(t_1, \beta)}{\overline{\alpha}}=(t_2 , 0)$ with $|\alpha|^2 -|\beta|^2=1$. This gives us that $\beta=0$. If $t_1 =0$, then $t_2 =0$. When at least one of $t_1, t_2$ is positive, the other one is also positive, we get $\alpha
	=1$ and we
	also get $t_1=t_2$. Thus the $t$ above is unique and consequently the $SU_{1,1} (\mathbb{C})$-orbits of $\mathbb{B}_2$ are $SU_{1,1} (\mathbb{C})\cdot (t,0),\,\,t\in [0,1)$.
	
	Now consider $SU_{1,1} (\mathbb{C})\cdot (0,0)=\{(0,\frac{\beta}{\overline{\alpha}}): \alpha,\beta\in \mathbb{C},|\alpha|^2 -|\beta|^2=1\}$. With a little effort we can show that the map $f:\D\rightarrow\mathbb{C}^2$ sending $z$ to $(0,z)$ maps $\D$ biholomorphically onto $SU_{1,1} (\mathbb{C})\cdot (0,0)$.
	
	Focusing on the remaining orbits, we find that for $t\in (0,1)$, $SU_{1,1} (\mathbb{C})\cdot (t,0)=\{(u,v)\in \mathbb{B}_2:|u|^2 + t^2 |v|^2=t^2\}$. Now it is easy to see that the biholomorphic map $g_t :\mathbb{C}^2\rightarrow \mathbb{C}^2$ sending $(z,w)$ to $(\frac{z}{t},w)$ maps $SU_{1,1} (\mathbb{C})\cdot (t,0)$ onto $S^3$, the topological boundary of $\mathbb{B}_2$. This completes the proof.
\end{proof}

Now let us focus on the action of $O_{2,1} (\mathbb{R})$. By $O_{2,1} (\mathbb{R})$, we mean the subgroup $\{A\in SU_{2,1} (\mathbb{C}): A \in M_3 (\mathbb{R}),
A^tI_{2,1}A=I_{2,1}\}$ of $SU_{2,1} (\mathbb{C})$. We only take a look at the $O_{2,1} (\mathbb{R})$-orbit of $(0,0)$.

\begin{proposition}\label{ActionOfO21}
	The $O_{2,1} (\mathbb{R})$-orbit of $(0,0)$ is totally real.
\end{proposition}
\begin{proof}
	If $A=\begin{pmatrix}
	a_{1} &a_{2} &a_{3}\\
	b_{1} &b_{2} &b_{3}\\
	c_{1} &c_{2} &c_{3}
	\end{pmatrix}\in O_{2,1}(\mathbb{R})$, then $A\cdot (0,0)=(\frac{a_3}{c_3},\frac{b_3}{c_3})$. So $O_{2,1} (\mathbb{R})\cdot (0,0)\subset \{(u,v)\in \mathbb{B}_2: Im(u)=0=Im(v)\}$. Suppose $(z,w)\in \mathbb{B}_2$ with $Im(z)=0=Im(w)$. If $(z,w)=(0,0)$, then we have $(z,w)=I_3 \cdot (0,0)$. If $(z,w)\neq (0,0)$, then consider the matrix
	
	\begin{align*}
	B=\begin{pmatrix}
	-\alpha w & k z &\gamma z\\
	\alpha z & k w & \gamma w\\
	0 & k (z^2 + w^2) & \gamma
	\end{pmatrix}
	\end{align*}
	where $\alpha =\frac{1}{\sqrt{z^2 + w^2}}, k=\frac{1}{\sqrt{(z^2 + w^2)(1-z^2 - w^2)}}$ and $\gamma =\frac{1}{\sqrt{1-z^2 - w^2}}$. It is easy to see that $B\in O_{2,1} (\mathbb{R})$ and $B\cdot (0,0)=(z,w)$. Hence we get that $O_{2,1} (\mathbb{R})\cdot (0,0)= \{(u,v)\in \mathbb{B}_2: Im(u)=0=Im(v)\}$. This proves the proposition.
\end{proof}

\section{Main result}\label{MainResult}

Before stating our main result, we want to remind the reader that given a 2-dimensional hyperbolic manifold $M$ with a 3-dimensional subgroup $G$ of $Aut(M)^0$, there are only a few possibilities for $dim_\mathbb{R} Aut(M)^0$. Obviously, $dim_\mathbb{R} Aut(M)^0 \geq
3$. It is a standard result (see, for example, Section~1.2 of 
\cite{IsaevLecAutGrps} and, in particular, equation (1.3) therein)
that for an $n$-dimensional hyperbolic complex manifold $M$, 
$dim_{\mathbb{R}}Aut(M) \leq n^2+2n$. In our case, this gives us: 
$dim_{\mathbb{R}}Aut(M) \leq 8$. Furthermore, we may conclude, using
Corollary~2.5 of \cite{IsaevLecAutGrps}, that $dim_{\mathbb{R}}Aut(M)
\neq 5$ (see also Theorem~2.2 therein). Finally, we may conclude, using Theorem~1.3 of
\cite{IsaevLecAutGrps}, that $dim_{\mathbb{R}}Aut(M) \neq 7$. This
shows that $dim_{\mathbb{R}}Aut(M)$, or, equivalently, 
$dim_{\mathbb{R}}Aut(M)^0$ is one of $3,4,6$ and $8$. When $dim_\mathbb{R} Aut(M)^0=8$, $M$ is biholomorphic with $\mathbb{B}_2$ 
(see, for example, Theorem~1.1 of \cite{IsaevLecAutGrps})
and when $dim_\mathbb{R} Aut(M)^0=6$, $M$ is biholomorphic with $\D^2$ (see, for example, Corollary~2.5 of \cite{IsaevLecAutGrps}). The cases of $3$ and $4$ are more complicated and can be found in \cite{Isaev_n_n2} and \cite{Isaev23}. We shall show that under certain conditions, found by observing the action of $G_\D$ on $\D^2$, the only possibility is $dim_\mathbb{R} Aut(M)^0=6$.

\begin{theorem}\label{CharacterizationOfD2}
	Let $M$ be a hyperbolic 2-manifold with a 3-dimensional non-solvable subgroup $G$ of $G(M)=Aut(M)^0$. Suppose that,
	under the action of $G$, the following conditions are met.
	
	\begin{enumerate}
		\item Exactly one $G$-orbit is a complex curve.
		\item The remaining $G$-orbits are strongly pseudoconvex real hypersurfaces.
		\item There exist a sequence $\{a_n\} \subset (0,1)$ converging to $1$ and, for each $n$, a $G$-orbit $\mathcal{O}_n$ in $M$ which is CR-equivalent to $\mathscr{F}_{a_n}$.
	\end{enumerate}
	Then $M$ is biholomorphic with the bidisc $\D^2$.
\end{theorem}

\begin{proof}
	All we have to do is to discard the cases $dim_\mathbb{R}
	G(M)=3,4,8$. Once this is done, it will follow that $M$ is
	biholomorphic with $\mathbb{D}^2$, 
	because, as mentioned above, it is known that the bidisc is,
	up to biholomorphic equivalence, the unique 2-dimensional
	hyperbolic complex manifold with 6-dimensional automorphism
	group.\\

	\textit{Case (i):} $dim_\mathbb{R} G(M)\neq 3$.\\
	Assume, to get a contradiction, that $dim_\mathbb{R} G(M)=
	3$. Then using Theorem \ref{CharOfD^2_r}, we conclude that
	$M$ is biholomorphic with $\D^2 _r$ for some $r\in (0,1)$.
	But the discussion following
	the proof of Theorem \ref{CharOfD^2_r} forbids this
	possibility. So $dim_\mathbb{R} G(M)\neq 3$.\\

	\textit{Case (ii):} $dim_\mathbb{R} G(M)\neq 4$.\\
	Assume, to get a contradiction, that $dim_\mathbb{R} G(M)=4$. 
	Then $M$ cannot be homogeneous, otherwise by Theorem~$1.1$ in \cite{Isaev_n_n2}, $M$ must be of dimension $3$ or $4$. By Proposition 2.1. in \cite{Isaev_n_n2}, each $G(M)$-orbit is either a real hypersurface or a complex curve. Clearly our manifold $M$ contains at most one complex curve and some real 3-dimensional hypersurfaces as $G(M)$-orbits. The $G(M)$-orbits have a connection with the $G$-orbits. If $O_{G(M)} (p)$ and $O_G (p)$ denote the $G(M)$-orbit and the $G$-orbit of $p\in M$, respectively, then for any $q\in O_{G(M)} (p)$, we must have $O_G (q)\subset O_{G(M)} (p)$. Also $q\in O_G (q)$ for all $q\in M$. Consequently $O_{G(M)} (p)=\bigcup_{q\in O_{G(M)} (p)} O_G (q)$ for all $p\in M$. This implies, in particular, that there is a 3-dimensional $G(M)$-orbit
	$\mathcal{O}$ that contains a
	$G$-orbit $\widetilde{\mathcal{O}}$
	CR-equivalent to some $\mathscr{F}_{a}$. Since both 
	$\mathcal{O}$ and $\widetilde{\mathcal{O}}$ are 3-dimensional,
	it follows that $\widetilde{\mathcal{O}}$ is an open subset
	of $\mathcal{O}$. 
	Now by Proposition 2.1. in \cite{Isaev_n_n2}, $\mathcal{O}$ is either Levi flat or spherical while each $\mathscr{F}_{a_m}$ is a 3-dimensional hypersurface which is strongly pseudoconvex and non-spherical (recall, from
	Proposition~\ref{CREquivalenceOfOrbits}, that 
	$\mathscr{F}_{a_m}$ is CR-isomorphic to 
	$\eta^{(2)}_{\alpha}$ for some $\alpha$, and 
	$\eta^{(2)}_{\alpha}$ is strongly pseudoconvex and 
	non-spherical---see \cite{Cartan}, \cite{IsaevAnalogues},
	\cite{Isaev23}). 
	In particular,
	$\mathcal{O}$, which is either Levi-flat or spherical, has a
	non-empty open subset, $\widetilde{\mathcal{O}}$, which is
	strongly pseudoconvex and non-spherical.
	This is a contradiction. Hence $dim_\mathbb{R} G(M)\neq 4$.\\

	\textit{Case (iii):} $dim_\mathbb{R} G(M)\neq 8$.\\
	Assume, to get a contradiction, that $\dim_{\mathbb{R}} G(M)
	= 8$. Then, by Theorem~1.1 of \cite{IsaevLecAutGrps}, $M$ is
	biholomorphic to $\mathbb{B}_2$. Therefore it follows from the
	hypotheses that 
	$G(\mathbb{B}_2)$ has a certain non-solvable 3-dimensional 
	Lie subgroup $G$ that satisfies the following properties: (1)
	the
	orbits of $\mathbb{B}_2$ under 
	$G$, except only one, are strongly pseudoconvex 3-dimensional
	real hypersurfaces and the sole remaining orbit is a complex
	curve, and (2)
	there exist an $a \in (0,1)$ and a 3-dimensional orbit of
	$\mathbb{B}_2$ under $G$ that is CR-equivalent to $\mathscr{F}_a$. 
	Now the 
	automorphism group of $\mathbb{B}_2$ is known (see Proposition~3,
	Section~7, Chapter~2 of \cite{Akhiezer}) to be 
	\[
	PSU_{2,1}(\mathbb{C}) =SU_{2,1}(\mathbb{C}) / K,
	\]
	where $K = \{ \lambda I_3 \mid \lambda^3 = 1 \}$.
	The Lie algebra of $G$
	is, by our assumption, non-solvable. 
	Now the conjugacy 
	classes of all the Lie subgroups of $SU_{2,1}(\mathbb{C})$ have been 
	determined in \cite{PWZ} and listed in
	Table~II on page 1390 of 
	\cite{PWZ}. Our group $G$ is (assumed to be isomorphic to)
	a 
	Lie subgroup of $SU_{2,1}(\mathbb{C})/K$ and must, therefore, be 
	of the
	form $N/K$, where $N$ is some Lie subgroup of $SU_{2,1}(\mathbb{C})$. Thus
	$N$ can be identified
	with a group of matrices that is a subgroup of $SU_{2,1}(\mathbb{C})$.
	Note that $SU_{2,1}(\mathbb{C})$ itself acts on $\mathbb{B}_2$,
	as we have discussed in subsection~\ref{ActionOnTheBall}.
	Note further that $N$ is a subgroup of $SU_{2,1}(\mathbb{C})$ containing
	$K$,
	and that the orbits of $\mathbb{B}_2$ under $N$ are precisely the
	orbits of $\mathbb{B}_2$ under $N/K$ (this is just a reflection of
	the general fact that if we take any two matrices $A$ and $B$
	in
	$SU_{2,1}(\mathbb{C})$ such that there exists $\lambda \in \C$ such that $B
	= 
	\lambda A$, then ($\lambda^3=1$) and for every $\underline{z} \in
	\mathbb{B}_2$, $A \cdot \underline{z} = B \cdot \underline{z}$, where $\cdot$ denotes the action of 
	$SU_{2,1}(\mathbb{C})$ on $\mathbb{B}_2$ that was
	discussed in subsection~\ref{ActionOnTheBall}). Now $N$ must
	belong to some conjugacy class appearing in
	the aforementioned table. Since $G$ is isomorphic to $N/K$,
	and $G$ is assumed to be non-solvable, $N$ must 
	be non-solvable as well. Therefore, from the table, we see
	that $N$ must belong either to the conjugacy class of
	$SU_{1,1}(\mathbb{C})$ or to that of $SU_2 (\mathbb{C})$ or to that of $O_{2,1}(\mathbb{R})$ (that is, isomorphic to some embedded copy as we have discussed in subsection~\ref{ActionOnTheBall}).
	It
	cannot belong to the conjugacy class of $SU_2 (\mathbb{C})$, because,
	if it did, then $G$ would be compact, which cannot be the
	case (by assumption, it acts on $\mathbb{B}_2$
	to produce non-compact orbits\,---\,$\mathscr{F}_a$ is non-compact).
	Consequently, $N$ must belong either to the conjugacy class of
	$SU_{1,1}(\mathbb{C})$ or to that of $O_{2,1}(\mathbb{R})$. 
	
	Suppose first that $N$ belongs
	to the conjugacy class of $SU_{1,1}(\mathbb{C})$. Consider the
	orbits of $\mathbb{B}_2$ under $N$. It has been shown
	that the 3-dimensional orbits of $\mathbb{B}_2$ under $SU_{1,1}(\mathbb{C})$ are
	spherical hypersurfaces (Proposition \ref{ActionOfSU11}). Our assumption implies that the orbits of
	$\mathbb{B}_2$ under $N/K$ are,
	with one exception, strongly pseudoconvex hypersurfaces, that
	the sole remaining orbit is a complex curve, and that there
	exist
	$a \in (0,1)$ and a hypersurface orbit that is CR-equivalent
	to $\mathscr{F}_a$. By the observation about the equality of the
	orbits of $\mathbb{B}_2$ under $N$ and $N/K$, it follows that there
	exists an orbit $\mathcal{O}$
	of $\mathbb{B}_2$ under $N$ that is CR-equivalent
	to $\mathscr{F}_{a}$ for some $a \in (0,1)$.
	Now, since $N$ belongs to the
	conjugacy class of $SU_{1,1}(\mathbb{C})$, it follows that $\mathcal{O}$ is the image
	of a spherical hypersurface in $\mathbb{B}_2$ under a holomorphic
	automorphism of $\mathbb{B}_2$ and hence is spherical itself.
	But that is a contradiction, because $\mathcal{O}$ is
	CR-equivalent to $\mathscr{F}_a$, which is not spherical.
	
	Suppose now that $N$ belongs to the conjugacy class of 
	$O_{2,1}(\mathbb{R})$. It has been shown that the orbit
	of $(0,0)$ under $O_{2,1}(\mathbb{R})$ is a totally real 2-dimensional
	submanifold of $\mathbb{B}_2$ (Proposition \ref{ActionOfO21}). Consequently, among the orbits
	of $\mathbb{B}_2$ under $N$, there must exist a totally real
	2-dimensional submanifold. But there does not, because, as 
	above, it follows from our assumption about the orbits of
	$\mathbb{B}_2$ under $N/K$ that the only 2-dimensional orbit 
	of $\mathbb{B}_2$ under $N$ is a complex curve.
	
	Therefore, we invariably get a contradiction. This shows
	that $\dim_{\mathbb{R}} G(M) \neq 8$.
	
	Thus, $dim_\mathbb{R} G(M)=6$, and this completes the proof.

\end{proof}

\vspace{0.1in} \noindent\textbf{Acknowledgements:}
The first-named author's research is supported by the University Grants Commission Centre for Advanced Studies. 
The second-named author's research is supported by a postdoctoral fellowship from the Harish-Chandra Research
Institute (Homi Bhabha National Institute).

\end{document}